\documentclass[a4paper]{amsart}

\usepackage[OT2,T1]{fontenc}
\DeclareSymbolFont{cyrletters}{OT2}{wncyr}{m}{n}
\DeclareMathSymbol{\Sha}{\mathalpha}{cyrletters}{"58}

\usepackage{amssymb,amscd}
\usepackage{mathrsfs}

\usepackage[
backref,
pdfauthor={Han Wu},  % add other authors
pdftitle={Non-invariance of weak approximation with Brauer--Manin obstruction},
]{hyperref}
\usepackage[alphabetic,backrefs,lite]{amsrefs}  % for bibliography, can choose nobysame

\theoremstyle{plain}
\theoremstyle{definition}

\newtheorem{theorem}{Theorem}[subsection]
\newtheorem{thm}{Theorem}[subsection]
\newtheorem{lemma}[theorem]{Lemma}

\newtheorem{proposition}[theorem]{Proposition}
\newtheorem{prop}[theorem]{Proposition}

\theoremstyle{definition}
\newtheorem{definition}[theorem]{Definition}

\newtheorem{qu}[theorem]{Question}

\newtheorem{eg}[theorem]{Example}

\newtheorem{conjecture}[theorem]{Conjecture}

\theoremstyle{remark}
\newtheorem{remark}[theorem]{Remark}

\newtheorem{recall}[theorem]{Recall}

\renewcommand{\AA}{\mathbb{A}}

\newcommand{\QQ}{\mathbb{Q}}

\newcommand{\ZZ}{\mathbb{Z}}

\newcommand{\FF}{\mathbb{F}}
\newcommand{\GG}{\mathbb{G}}
\newcommand{\PP}{\mathbb{P}}

\newcommand{\Ecal}{{\mathcal E}}
\newcommand{\Lcal}{{\mathcal L}}
\newcommand{\Ocal}{{\mathcal O}}

\newcommand{\Lscr}{{\mathscr L}}

\newcommand{\half}{1/2}

\newcommand{\invap}{\inv_v(A(P_v))}
\newcommand{\invaq}{\inv_v(A(Q_v))}
\newcommand{\invaqo}{\inv_v(A(Q_0))}

\DeclareMathOperator{\inv}{inv}
\DeclareMathOperator{\Proj}{Proj}
\DeclareMathOperator{\Sym}{Sym}
\DeclareMathOperator{\Spec}{Spec}

\newcommand{\defi}[1]{\textsf{#1}} % for defined terms

\newcommand{\Br}{\textup{Br}}
\newcommand{\et}{\textup{\'et}}

\makeatletter
\g@addto@macro\bfseries{\boldmath}  % This makes math in section titles bold.
\makeatother

\begin{document}
	
	\begin{title}
		{Non-invariance of weak approximation with Brauer--Manin obstruction}  %\'etale
	\end{title}
	\author{Han Wu}
	\address{University of Science and Technology of China,
		School of Mathematical Sciences,
		No.96, JinZhai Road, Baohe District, Hefei,
		Anhui, 230026. P.R.China.}
	\email{wuhan90@mail.ustc.edu.cn}
	\date{}
	%\thanks{The author was partially supported by USTC}
	% 11G05, , 14H25, 14H52, 14K15, 14J30

	%\thanks{The authors were partially supported by University of Science and Technology of China}
	%\thanks{\textit{MSC 2010} : 11G35 14G05  14G25 14J20}

	% % % ----------------------------------------------------------------------

	% % % ----------------------------------------------------------------------

	\begin{abstract} 
		In this paper, we study weak approximation with Brauer--Manin obstruction with respect to extensions of number fields. For any nontrivial extension $L/K,$ assuming a conjecture of M. Stoll, we prove that there exists a $K$-threefold satisfying  weak approximation with Brauer--Manin obstruction off all archimedean places, while its base change to $L$ fails.
		Then we illustrate this construction with an explicit unconditional example.
	\end{abstract} 
	
	\maketitle

	\section{Introduction}
	
	\subsection{Background}
	Let $X$ be a proper algebraic variety (\cite[Definition P. 105]{Ha97}) defined over a number field $K.$ Let $\Omega_K$ be the set of all nontrivial places of $K.$ Let $\infty_K\subset \Omega_K$ be the subset of all archimedean places, and let $S\subset \Omega_K$ be a finite subset. Let $K_v$ be the completion of $K$ at $v\in \Omega_K.$ Let $\AA_K$ be the ring of ad\`eles of $K.$ 
	We assume that $X(K)\neq\emptyset.$  Let $pr^S\colon\AA_K\to \AA_K^S$ be the natural projection of ad\`eles to ad\`eles without $S$ components, which induces a natural projection $pr^S\colon X(\AA_K)\to X(\AA_K^S).$ Since $X$ is proper, the set of adelic points $X(\AA_K^S)$ is equal to $\prod_{v\in \Omega_K\backslash S}X(K_v),$ and the adelic topology of $X(\AA_K^S)$  is indeed the product of  $v$-adic topologies. %There is a natural diagonal embedding of the set of rational points $X(K)$ into the set of adelic points $X(\AA_K^S).$ 
	Viewing $X(K)$ as a subset of $X(\AA_K)$ (respectively of $X(\AA_K^S)$) by the diagonal embedding, we say that $X$
	satisfies \defi{weak approximation} (respectively \defi{weak approximation off $S$}) if $X(K)$ is dense in $X(\AA_K)$ (respectively in $X(\AA_K^S)$), cf. \cite[Chapter 5.1]{Sk01}. 
	
	Cohomological obstructions have been used to explain the failure of density of $X(K)$ in $X(\AA_K^S).$  Let $\Br(X)=H^2_{\et}(X,\GG_m)$ be the Brauer group of $X.$ The Brauer--Manin pairing
		\begin{align*}
		X(\AA_K)\times\Br(X)&\to \QQ/\ZZ\\
		(A,\{P_v\}_{v\in \Omega_K})&\mapsto \sum\limits_{v\in\Omega_K}\inv_v(A(P_v)),
	\end{align*}
	suggested by Manin \cite{Ma71}, is provided by local class field theory. The left kernel of this pairing is denoted by $X(\AA_K)^{\Br},$ which is a closed subset of $X(\AA_K).$ By the global reciprocity in class field theory, there is an exact sequence: 
	$$0\to \Br(K)\to \bigoplus\limits_{v\in \Omega_K} \Br(K_v)\to \QQ/\ZZ\to 0.$$
	It induces an inclusion: $X(K)\subset pr^S(X(\AA_K)^\Br).$ 
	We say that $X$ satisfies \defi{weak approximation with Brauer--Manin obstruction} (respectively \defi{with Brauer--Manin obstruction off $S$}) if $X(K)$ is dense in $X(\AA_K)^\Br$ (respectively in $pr^S(X(\AA_K)^\Br)$).
	
	 For an elliptic curve over $\QQ$ of analytic rank one, it satisfies weak approximation with Brauer--Manin obstruction, cf. \cite{Wa96}. For an abelian variety defined over $K,$ if its Tate--Shafarevich group is finite, then it satisfies weak approximation with Brauer--Manin obstruction off $\infty_K,$ cf. \cite[Proposition 6.2.4]{Sk01}. 
	For any smooth, proper and rationally connected variety defined over a number field, it is conjectured by Colliot-Th\'el\`ene \cite{CT03} that it satisfies weak approximation with Brauer--Manin obstruction. Colliot-Th\'el\`ene's conjecture holds for Ch\^atelet surfaces, cf. \cite{CTSSD87a,CTSSD87b}.
	For any smooth, projective, and geometrically connected curve defined over a number field $K,$ it is conjectured by Stoll \cite[Conjecture 9.1]{St07} that it satisfies weak approximation with Brauer--Manin obstruction off $\infty_K\colon$ see Conjecture \ref{conjecture Stoll} for more details.

	\subsection{Question}\label{Question}
	Given a nontrivial extension of number fields $L/K,$ and a finite subset $S\subset\Omega_K,$ let $S_L\subset \Omega_L$ be the subset of all places above $S.$ Let $X$ be a smooth, projective, and geometrically connected variety defined over $K.$ Let $X_L=X\times_{\Spec K} {\Spec L}$ be the base change of $X$ to $L.$ In this paper, we consider the following question.
	\begin{qu}\label{question on WA1}
		If a smooth, projective, and geometrically connected variety $X$ has a $K$-rational point, and satisfies weak approximation with Brauer--Manin obstruction off $S,$ must $X_L$ also satisfy weak approximation with Brauer--Manin obstruction off $S_L?$	
	\end{qu}

	\subsection{A negative answer to Question \ref{question on WA1}} 
	For any number field $K$ and some quadratic extension $L/K,$ assuming \cite[Conjecture 9.1]{St07}, a Ch\^atelet surface bundle over a curve was constructed by Liang\cite{Li18}, which gives a negative answer to Question \ref{question on WA1}.
	Also an unconditional example with explicit equations was given for $K=\QQ$ and $L=\QQ(\sqrt{5})$ in loc. cit. We generalize Liang's work to any nontrivial extension of number fields in this paper.
	
	For any nontrivial extension of number fields $L/K,$ assuming \cite[Conjecture 9.1]{St07}, we have the following theorem to give a negative answer to Question \ref{question on WA1}.
	\begin{thm}[Theorem \ref{theorem main result: non-invariance of weak approximation with BMO}]
		For any nontrivial extension of number fields $L/K,$ and any finite subset $T\subset \Omega_L,$ assuming \cite[Conjecture 9.1]{St07}, there exist a curve $C$ and a Ch\^atelet surface bundle: $X\to C$ defined over $K$ such that
		\begin{itemize}
			\item $X$ has a $K$-rational point, and satisfies weak approximation with Brauer--Manin obstruction %, \'etale Brauer--Manin obstruction
			off $\infty_K,$
			\item $X_L$ does not satisfy weak approximation with Brauer--Manin obstruction %, \'etale Brauer--Manin obstruction
			off $T.$
		\end{itemize}
	\end{thm}
	
	For $K=\QQ$ and $L=\QQ(\sqrt{3}),$ based on the method given in Theorem \ref{theorem main result: non-invariance of weak approximation with BMO}, we give an explicit unconditional example  in Section \ref{section main example1}. One can use Liang's method to construct an explicit unconditional example satisfying Theorem \ref{theorem main result: non-invariance of weak approximation with BMO} for $\QQ(\sqrt{3})/\QQ.$ But, we use different method to prove Theorem \ref{theorem main result: non-invariance of weak approximation with BMO} and construct an explicit unconditional example.
	The $3$-fold $X$ is a smooth compactification of the following $3$-dimensional affine subvariety in $\AA^5$ with affine coordinates $(x,y,z,x',y')$ given by equations 
		\begin{equation*}
		\begin{cases}
			y^2-73z^2=(1-x^2)(x^2-73)(x'-4)^2+(5x^2+1)(5334x^2/5329+1/5329)\\ 
			{y'}^2={x'}^3-16
		\end{cases}.
	\end{equation*}

	% Although for these two fields, using the method given in \cite[Theorem 4.5]{Li18},  man can construct an explicit unconditional example having the properties of Theorem \ref{theorem main result: non-invariance of weak approximation with BMO}, the method that we use is different.
	
	%	Let $\zeta_7$ be a primitive $7$-th root of unity.
	%	For $K=\QQ$ and $L=\QQ(\zeta_7+\zeta_7^{-1}),$ based on the method given in Theorem \ref{theorem main result: non-invariance of weak approximation with BMO}, we give an explicit unconditional example  in Subsection \ref{section main example1}.  The method given in \cite[Theorem 4.5]{Li18} no longer works. The $3$-fold $X$ is a smooth compactification of the following $3$-dimensional affine subvariety given by equations 
	%	\begin{equation*}
	%		\begin{cases}
	%			
	%			
	%			
	%			
	%			
	%			
	%			
	%			
	%			y^2-73z^2=(1-x^2)(x^2-73)(x'-4)^2+(99x^2+1)(\frac{5428}{5329}x^2+\frac{1}{5329})\\ 
	%			{y'}^2={x'}^3-16
	%		\end{cases}
	%	\end{equation*}
	%	in $\AA^5$ with affine coordinates $(x,y,z,x',y').$
	%	

%	\subsubsection{Main ideas behind our construction in the proof of Theorem \ref{theorem main result: non-invariance of weak approximation with BMO}}
We explain the idea of proving  Theorem \ref{theorem main result: non-invariance of weak approximation with BMO}.
	Given a nontrivial extension of number fields $L/K,$
	we start with a curve $C$ such that $C(K)$ and $C(L)$ are both finite, nonempty and $C(K)\neq C(L).$ Combining arithmetic of Ch\^atelet surfaces with a construction method from Poonen \cite{Po10}, we choose two Ch\^atelet surfaces denoted by $V_0$ and $V_\infty$ over $K,$ and construct  a Ch\^atelet surface bundle: $\beta \colon X\to C$ such that the fiber of each point in $C(K)$ is isomorphic to $V_\infty,$ and that the fiber of each point in $C(L)\backslash C(K)$ is isomorphic to $V_{0L}.$ 
	Roughly speaking, $X(\AA_K)^\Br$ (respectively $X_L(\AA_L)^\Br$) is the union of adelic points of rational fibers by  \cite[Proposition 5.4]{Po10} together with \cite[Conjecture 9.1]{St07}. Using  fibration methods, the arithmetic properties of $V_\infty$ and $V_0$ will determine that of $X.$ 
	
%	, and make full use of this real place information. 	
%	More exactly, we choose a Ch\^atelet surface $V_0$ such that it has local points in all but the given real place, and choose a Ch\^atelet surface $V_\infty$ such that it has local points in all but some splitting finite place. This will make sure $X(L)=\emptyset,$ $X(\AA_K)^{\Br}=\emptyset,$ but $X(\AA_K)\neq\emptyset.$ Then
%	we modify an adelic point attached to one $C(L)\backslash C(K)$ point at the real place by one $C(K)$ point, 	
%	 $X(\AA_K)\neq \emptyset.$\\
%	By the same argument as in the proof of Theorem \ref{theorem main result: non-invariance of the Hasse principle with BMO for odd degree}, the set $X(\AA_K)^{\Br}=\emptyset.$

	\section{Notation and preliminaries}
\subsection{Notation}	Given a number field $K,$ let $\Ocal_K$ be the ring of its integers, and let $\Omega_K$ be the set of all its nontrivial places. Let $\infty_K\subset \Omega_K$ be the subset of all archimedean places, and let $2_K\subset \Omega_K$ be the subset of all $2$-adic places. Let $\infty_K^r\subset \infty_K$ be the subset of all real places. Let $\Omega_K^f=\Omega_K\backslash \infty_K$ be the set of all finite places of $K.$ An \defi{odd place} will mean a finite place not contained in $2_K.$	
	Let $K_v$ be the completion of $K$ at  $v\in \Omega_K.$ %and let $\AA_K$ be the ring of ad\`eles of $K.$ 	
	%If an element $a\in \Ocal_K$ is a prime element, we denote its prime ideal by $\pfr_a$ and its associated valuation by $v_a.$
	For $v\in \infty_K,$ let $\tau_v\colon K\hookrightarrow K_v$ be the embedding of $K$ into its completion. For $v\in \Omega_K^f,$ let $\Ocal_{K_v}$ be its valuation ring, and let $\FF_v$ be its residue field. %and $\pfr_v$ be the prime ideal associated to $v.$  
	We say that an element is a \defi{prime element}, if the ideal generated by this element is a prime ideal. For a prime element $p\in \Ocal_K,$ we denote %its prime ideal by $\pfr_a,$ and denote 
	its associated place by $v_p.$
	 Let $\AA_K$ (respectively $\AA_K^S$) be the ring of ad\`eles (ad\`eles without $S$ components) of $K.$

	\subsection{Hilbert symbol}

	 We use the Hilbert symbol $(a,b)_v\in \{\pm 1\},$ for any $a,b\in K_v^\times$ and $v\in \Omega_K.$  By definition, $(a,b)_v=1$ if and only if $x_0^2-ax_1^2-bx_2^2=0$ has a $K_v$-solution in $\PP^2$ with homogeneous coordinates $(x_0:x_1:x_2),$ which equivalently means that the curve defined over $K_v$ by the equation $x_0^2-ax_1^2-bx_2^2=0$ in $\PP^2,$ is isomorphic to $\PP^1.$ The Hilbert symbol gives a symmetric bilinear form on $K_v^\times/K_v^{\times 2}$ with value in $\ZZ/2\ZZ,$ cf. \cite[Chapter XIV, Proposition 7]{Se79}. 
	
 We state the following lemmas for the Hilbert symbol. The first two lemmas have already been given in the paper \cite{Wu22a}. We give the statement here for the convenience of the reader.

	\begin{lemma}\label{lemma hilbert symbal lifting for odd prime}
		Let $K$ be a number field, and let $v$ be an odd place of $K.$ Let $a, b\in K_v^\times$ such that $v(a),v(b)$ are even. Then $(a,b)_v=1.$
	\end{lemma}
	
%	\begin{proof}
%		Choose a prime element $\pi_v\in K_v.$ Let $a_1=a\pi_v^{-v(a)}$ and $b_1=b\pi_v^{-v(b)}.$ Since the valuations $v(a)$ and $v(b)$ are even, the elements $\pi_v^{-v(a)}$ and $\pi_v^{-v(b)}$ are in $K_v^{\times 2}.$ So $(a,b)_v=(a_1,b_1)_v$ and $a_1,b_1\in \Ocal_{K_v}^\times.$
%		By Chevalley--Warning theorem (cf. \cite[Chapter I \S 2, Corollary 2]{Se73}), the equation $x_0^2-\bar{a}_1x_1^2-\bar{b}_1x_2^2=0$ has a nontrivial solution in $\FF_v.$ Since $v$ is odd, by Hensel's lemma, this solution can be lifted to a nontrivial solution in $\Ocal_{K_v}.$ Hence $(a,b)_v=(a_1,b_1)_v=1.$
%	\end{proof}
	
	\begin{lemma}\label{lemma Hensel lemma for Hilbert symbal}
		Let $K$ be a number field, and let $v$ be an odd place of $K.$ Let $a,b,c\in K_v^\times$ such that $v(b)<v(c).$ Then $(a,b+c)_v=(a,b)_v.$
	\end{lemma}
	
%	\begin{proof}
%		For $v(b)<v(c),$ we have $v(b^{-1}c)>0.$ By Hensel's lemma, we have $1+b^{-1}c\in K_v^{\times 2}.$ So $(a,b+c)_v=(a,b(1+b^{-1}c))_v=(a,b)_v.$
%	\end{proof}
	
		\begin{lemma}\label{lemma (p,-1)=1}
		Let $K$ be a number field. Let $p$ be its prime element such that 
		\begin{itemize}
			\item $\tau_v(p)>0$ for all $v\in \infty_K^r,$
			\item $p\equiv 1 \mod 8\Ocal_K.$
		\end{itemize}
		Then $(p,-1)_{v_{p}}=1.$
	\end{lemma}
	\begin{proof}
		The condition that $p\equiv 1 \mod 8\Ocal_K$ implies $p\in K_v^{\times 2}$ for all $v\in 2_K.$ By the product formula $\prod_{v\in \Omega_K} (p,-1)_v=1$ and Lemma \ref{lemma hilbert symbal lifting for odd prime}, this lemma follows.
	\end{proof}

		\subsection{Poonen's proposition}

	Since our result is based on  \cite[Proposition 5.4]{Po10}. We recall Poonen's general result. There exist some remarks on it in \cite[Section 4.1]{Li18}. Colliot-Th\'el\`ene \cite[Proposition 2.1]{CT10} gave another proof of his proposition.

	\begin{recall}\label{recall of Poonen's notation}
		Let $B$ be a smooth, projective, and geometrically connected variety over a number field $K.$ Let $\Lcal$ be a line bundle on $B$ such that the set of global sections $\Gamma(B,\Lcal^{\otimes2})\neq 0.$ Let $\Ecal=\Ocal_B\oplus\Ocal_B\oplus \Lcal.$ Let $a$ be a constant in $K^\times,$ and let $s$ be a nonzero global section in $\Gamma(B,\Lcal^{\otimes2}).$ The zero locus of $(1,-a,-s)\in\Gamma(B,\Ocal_B\oplus\Ocal_B\oplus\Lcal^{\otimes2})\subset\Gamma(B,\Sym^2\Ecal)$ in the projective space bundle $\Proj(\Ecal)$ is  a projective and geometrically integral variety, denoted by $X$ with the natural projection $X\to B.$ Let $\overline{K}$ be an algebraic closure of $K.$ Denote $B\times_{\Spec K} {\Spec \overline{K}}$ by $\overline{B}.$
	\end{recall}	
	\begin{prop}\cite[Proposition 5.3]{Po10}\label{Poonen's main proposition}
		Given a number field $K,$ we use the notation as in Recall \ref{recall of Poonen's notation}. Let $\alpha\colon X\to B$ be the natural projection.
		Assume that 
		\begin{itemize}
			\item the locus defined by $s=0$ in $B$ is smooth, projective, and geometrically connected,
			\item $\Br\overline{B}=0$ and $X(\AA_K)\neq \emptyset.$
		\end{itemize}
		Then $X$ is smooth, projective, and geometrically connected. Moreover, the group homomorphism $\alpha^*\colon\Br(B)\to \Br(X)$ is an isomorphism. 
	\end{prop}

	\section{Ch\^atelet surfaces}
	
	Let $K$ be a number field. \defi{Ch\^atelet surfaces} are defined to be smooth projective models of conic bundle surfaces defined by the equation
	\begin{equation}\label{equation}
		y^2-az^2=P(x)
	\end{equation} in $K[x,y,z]$ such that $a\in K^\times,$ and that $P(x)$ is a separable degree-$4$ polynomial in $K[x].$ Given an equation (\ref{equation}), let $V^0$ be the affine surface in $\AA^3_K$ defined by this equation. Let $V$ be the natural smooth compactification  of $V^0$ given in \cite[Section 7.1]{Sk01}, which is called the Ch\^atelet surface given by this equation, cf. \cite[Section 5]{Po09}.
	
	\begin{remark}\label{remark birational to PP^2}
		For any local field $K_v,$  if $a\in K_v^{\times 2},$ then $V$ is birationally equivalent to $\PP^2$ over $K_v.$ By the implicit function theorem, there exists a $K_v$-point on $V.$
	\end{remark}
	
	\begin{remark}\label{remark the implicit function thm and local constant Brauer group}
		For any local field $K_v,$ by smoothness of $V,$ the implicit function theorem implies that the nonemptiness of $V^0(K_v)$ is equivalent to the nonemptiness of $V(K_v),$ and that $V^0(K_v)$ is open dense in $V(K_v)$ with the $v$-adic topology. Given an element $A\in\Br(V),$ the evaluation of $A$ on $V(K_v)$ is locally constant. By the properness of $V,$ the space $V(K_v)$ is compact. So the set of all possible values of the evaluation of $A$ on $V(K_v)$ is finite. Indeed, by \cite[Proposition 7.1.2]{Sk01}, there exist only two possible values. They are determined by the evaluation of $A$ on $V^0(K_v).$ In particular, if the evaluation of $A$ on $V^0(K_v)$ is constant, then it is constant on $V(K_v).$ 
	\end{remark}

	\begin{remark}\label{remark birational to Hasse--Minkowski theorem}
		If the polynomial $P(x)$ has a factor $x^2-a,$ i.e. there exists a degree-$2$ polynomial $f(x)$ such that $P(x)=f(x)(x^2-a),$ then $Y=\frac{xy+az}{x^2-a}$ and $Z=\frac{y+xz}{x^2-a}$ give a birational equivalence between $V$ and a quadratic surface given by $Y^2-aZ^2=f(x)$ with affine coordinates $(x,Y,Z).$ By the Hasse--Minkowski theorem, the surface $V$ satisfies weak approximation. 
	\end{remark}

%	\subsection{Ch\^atelet surfaces with rational points and not satisfying weak approximation}

	Given a number field $K,$ Liang \cite[Proposition 3.4]{Li18} constructed a Ch\^atelet surface over $K,$ which has a $K$-rational point and does not satisfy weak approximation off $\infty_K.$ From Liang's construction, there exists an element in the Brauer group of this surface, which has two different local invariants on a given finite place, i.e. this element gives an obstruction to weak approximation for this surface. In \cite[Proposition 4.1]{Wu22a}, the author generalized Liang's
	construction. Although it is enough for us, for the convenience of the reader, we use \v{C}ebotarev's density theorem, global class field theory and Dirichlet's theorem on arithmetic progressions to simplify it.

	We construct a Ch\^atelet surface explicitly having the following weak approximation property. It will be used in the proof of Theorem \ref{theorem main result: non-invariance of weak approximation with BMO}.

	\begin{proposition}(compare to \cite[Proposition 4.3]{Wu22a})\label{proposition the valuation of Brauer group on local points are fixed outside S and take two value on S}		
			For any extension of number fields $L/K,$ and any finite subset $T_0\subset \Omega_L,$
		there exists a Ch\^atelet surface $V_0$ defined over $K$ such that the set $V_0(K)\neq \emptyset,$ and that the surface $V_{0L}$ does not satisfy weak approximation off $T_0.$
	\end{proposition}

			The rest of this section is devoted to the proof of Proposition \ref{proposition the valuation of Brauer group on local points are fixed outside S and take two value on S}.
		We construct the Ch\^atelet surface $V_0$ explicitly.
		
	\subsubsection{Choosing elements for the parameters in the equation (\ref{equation})}\label{subsection choose an element a for S}
	
	By \v{C}ebotarev's density theorem and global class field theory applied to a ray class
	field, we can find a prime element $p_1\in \Ocal_K$ such that  
	\begin{itemize}
		\item $\tau_v(p_1)>0$ for all $v\in \infty_K^r,$
		\item $p_1\equiv 1 \mod 8\Ocal_K,$
		\item $p_1$ splits completely in $L,$
		\item $v'\nmid v_{p_1}$ for all $v'\in T_0.$
	\end{itemize}
	We refer to \cite[Lemma 2.0.1]{Wu21} for more details.
	By the generalised Dirichlet's theorem on
	arithmetic progressions, we find another prime element $p_2\in \Ocal_K$ such that $(p_1,p_2)_{v_{p_1}}=-1.$ We refer to \cite[Proposition 2.1]{Li18} for more details.
	
	Let $V_0$ be the Ch\^atelet surface given by $y^2-p_1z^2=(p_2x^2+1)((1+p_2/p_1^2)x^2+1/p_1^2).$

	\begin{lemma}(compare to \cite[Proposition 4.1]{Wu22a})\label{lemma: the valuation of Brauer group on local points are fixed outside S and take two value on S}
		The Ch\^atelet surface $V_0$ defined over $K,$ has the following properties.
		\begin{enumerate}
			\item The subset $V_0(K)\subset V_0(L)$ is nonempty. The natural maps $\Br(K)\to \Br(V_0)$ and $\Br(L)\to \Br(V_{0L})$	are injective. The Brauer group $\Br(V_0)/\Br(K)\cong\Br(V_{0L})/\Br(L)\cong \ZZ/2\ZZ,$ is generated by an element $A\in \Br(V_0).$ 
			\item For the place $v=v_{p_1},$ there exist $P_v$ and $Q_v$ in $V_0(K_v)$ such that the local invariants $\inv_v(A(P_v))=0$ and $\inv_v(A(Q_v))=\half.$  For any other $v\neq v_{p_1},$ and any $P_v\in V_0(K_v),$ the local invariant $\invap=0.$	
			\item For any $v'|v_{p_1},$ there exist $P_{v'}$ and $Q_{v'}$ in $V_0(L_{v'})$ such that the local invariants $\inv_{v'}(A(P_{v'}))=0$ and $\inv_{v'}(A(Q_{v'}))=\half.$  For any other $v'\nmid v_{p_1},$ and any $P_{v'}\in V_0(L_{v'}),$ the local invariant $\inv_{v'}(A(P_{v'})) =0.$
		\end{enumerate}
		\end{lemma}
			
	\begin{proof}
		\begin{enumerate}
			\item 		For $(x,y,z)=(0,1/p_1,0)$ is a rational point on $V_0^0,$ the set $V_0(K)$ is nonempty. We denote this rational point by $Q_0.$ And $Q_0$ gives sections of maps $\Br(K)\to \Br(V_0)$ and $\Br(L)\to \Br(V_{0L}),$ so they are injective.
						
			By the choice of $p_1$ and Lemma \ref{lemma (p,-1)=1}, we have $(p_1,-1)_{v_{p_1}}=1.$ By the choice of $p_2,$ 
			 we have $(p_1,p_2)_{v_{p_1}}=-1.$ So			 
			 $(p_1,-p_2)_{v_{p_1}}=(p_1,-p_2)_{v_{p_1}}(p_1,-1)_{v_{p_1}}=-1.$ Hence $-p_2\notin K_{p_1}^{\times 2}.$ The chosen condition that $p_1$ splits  in $L$ implies  $-p_2\notin L^{\times 2}.$ So the polynomial $p_2x^2+1$  is irreducible over $K$ and $L.$ 
			According to \cite[Proposition 7.1.1]{Sk01}, the Brauer group $\Br(V_0)/\Br(K)\cong\Br(V_{0L})/\Br(L)\cong \ZZ/2\ZZ.$ Furthermore, by Proposition 7.1.2 in loc. cit, we take the quaternion algebra $A=(p_1,p_2x^2+1)\in \Br(V_0)$ as a generator element of this group. Then we have the equality $A=(p_1,p_2x^2+1)=(p_1,(1+p_2/p_1^2)x^2+1/p_1^2)$ in $\Br(V_0).$

			\item 	
			For any $v\in \Omega_K,$ since $(p_1,1)_v=1,$ the local invariant $\invaqo=0.$ By Remark \ref{remark the implicit function thm and local constant Brauer group}, it suffices to compute the local invariant $\invap$ for all $P_v\in V_0^0(K_v).$
			\begin{enumerate}
				\item		Suppose that $v\in \infty_K^r\cup 2_K.$  Then $p_1\in K_v^{\times 2},$ so $\invap=0$ for all $P_v\in V_0(K_v).$
					\item Suppose that $v\in \Omega_K^f\backslash (v_{p_1}\cup 2_K ).$ Take an arbitrary $P_v\in V_0^0(K_v).$ If $\invap=1/2,$ then $(p_1,p_2x^2+1)_v=-1=(p_1,(1+p_2/p_1^2)x^2+1/p_1^2)_v$ at $P_v.$ For $v(p_1)=0,$  by Lemma \ref{lemma hilbert symbal lifting for odd prime}, the first equality implies that $v(p_2x^2+1)$ is odd. So $v(x)\leq 0.$ Hence $v(p_2+x^{-2})$ is odd and positive. For $v(p_1)=0,$ by Hensel's lemma, we have $1+(p_2+x^{-2})/p_1^2\in K_v^{\times 2}.$ So $(p_1,(1+p_2/p_1^2)x^2+1/p_1^2)_v=(p_1,x^2)_v(p_1,1+(p_2+x^{-2})/p_1^2)_v=1,$ which is a contradiction. So $\invap=0.$
					\item Suppose that $v=v_{p_1}.$ Take $P_v=Q_0,$ then $\invap=0.$ Take $x_0\in K_v$ such that  $v(x_0)< 0.$ For $v(p_2)=0,$ by Lemma \ref{lemma Hensel lemma for Hilbert symbal}, we have $(p_1,p_2x_0^2+1)_v=(p_1,p_2x_0^2)_v=(p_1,p_2)_v$ and $(p_1,(1+p_2/p_1^2)x_0^2+1/p_1^2)_v=(p_1,p_2x_0^2/p_1^2)_v=(p_1,p_2)_v.$ So
				$(p_1,(p_2x_0^2+1)((1+p_2/p_1^2)x_0^2+1/p_1^2))_v=(p_1,p_2)_v(p_1,p_2)_v=1.$ Hence, there exists a $Q_v\in V_0^0(K_v)$ with $x=x_0.$ For $(p_1,p_2)_v=-1,$ we have $\invaq=\half.$ 	
			\end{enumerate}

			\item 		For any $v'\in \Omega_L,$ the local invariant $\inv_{v'}(A(Q_0))=0.$\\
			Suppose that $v' |v_{p_1}.$  By the assumption that $v_{p_1}$ splits completely in $L,$ we have $K_{v_{p_1}}=L_{v'}.$ So $V_0(K_{v_{p_1}})=V_0(L_{v'}).$	By the argument already shown, there exist $P_v, Q_v\in V_0(K_v)$ such that $\inv_v(A(P_v))=0$ and $\inv_v(A(Q_v))=\half.$ View $P_v, Q_v$ as elements in $V_0(L_{v'}),$ and let $P_{v'}=P_v$ and $Q_{v'}=Q_v.$ Then $\inv_{v'}(A(P_{v'}))=\invap=0$ and $\inv_{v'}(A(Q_{v'}))=\invaq=\half.$ \\
			Suppose that $v'\nmid v_{p_1}.$ This local computation is the same as the case $v\in \Omega_K\backslash \{v_{p_1}\}.$	
		\end{enumerate}

	\end{proof}

	With the help of Lemma \ref{lemma: the valuation of Brauer group on local points are fixed outside S and take two value on S}, we now prove that the Ch\^atelet surface $V_0$ has the property of Proposition \ref{proposition the valuation of Brauer group on local points are fixed outside S and take two value on S}.

	\begin{proof} 
		 Let $T'\subset \Omega_L$ be the subset of all places above $v_{p_1}.$ We take a place $v_0'\in T'.$ Since $p_1$ splits completely in $L,$ we have $L_{v_0'} =K_{v_{p_1}}.$ 
		Let $U_{v_0'}=\{P_{v_0'}\in V_0(L_{v_0'})| \inv_{v_0'}(A(P_{v_0'}))=\half\}.$ For $v'\in T'\backslash \{v_0'\},$  let $U_{v'}=\{P_{v'}\in V_0(L_{v'})| \inv_{v'}(A(P_{v'}))=0\}.$ For any $v'\in T',$ by Lemma \ref{lemma: the valuation of Brauer group on local points are fixed outside S and take two value on S}, the set $U_{v'}$ is a nonempty  open subset of $V_0(L_{v'}).$ Let $M=\prod_{v'\in T'}U_{v'}\times \prod_{v'\notin T' }V_0(L_{v'}).$ It is a nonempty  open subset of $V_0(\AA_L).$ For any $(P_{v'})_{v'\in \Omega_L}\in M,$ by Lemma \ref{lemma: the valuation of Brauer group on local points are fixed outside S and take two value on S} and the choice of $U_{v'},$ the sum $\sum_{v'\in \Omega_L}\inv_{v'}(A(P_{v'}))=\half$ is nonzero in $\QQ/\ZZ.$ So $V_{0L}(\AA_L)^{\Br}\cap M=\emptyset,$ which implies  $V_0(L)\cap M= \emptyset.$ By the choice of $p_1,$ we have $v'\nmid v_{p_1}$ for all $v'\in T_0.$ Then $T_0\cap T'=\emptyset.$ Hence $V_{0L}$ does not satisfy weak approximation off $T_0.$
	\end{proof}

	Using the constructional method, we have the following example, which is a special case of Proposition \ref{proposition the valuation of Brauer group on local points are fixed outside S and take two value on S}. It will be used for further discussion.
	
	\begin{eg}\label{example1: construction of V_0}
		For $K=\QQ$ and $L=\QQ(\sqrt{3}),$ and  let $T_0\subset \Omega_L\backslash \{73$-adic places$\}$ be a finite subset. 
		We choose prime elements: $p_1=73$ splitting completely in $L$ and $p_2=5.$ Then the Ch\^atelet surface given by $y^2-73z^2=(5x^2+1)(5334x^2/5329+1/5329),$ has the properties of  Propositions \ref{proposition the valuation of Brauer group on local points are fixed outside S and take two value on S}.
	\end{eg}
	%\begin{remark}
	%	Let $S_1$ and $S_2$ be two finite sets in $\Omega_K$ and both $S_1$ and $S_2$ satisfy the conditions of Proposition \ref{corollary the valuation of Brauer group on local points are fixed and on S_L nontrivial}, then there exist $a\in K\backslash L^2$ and relatively prime separable degree-4 polynomials $P_1$ and $P_2$ such that the Ch\^atelet surfaces given by $y^2-az^2=P_1$ and $y^2-az^2=P_1$ $W_1$ and $W_2$ 
	%\end{remark}

	\section{Stoll's conjecture for curves}
	For a smooth, projective, and geometrically connected curve defined over a number field, Stoll \cite[Conjecture 9.1]{St07} made the following conjecture. 
	
	Given a curve $C$ defined over a number field $K,$ let $C(\AA_K)_\bullet=\prod_{v\in \infty_K}\{$connected components of $C(K_v)\}\times C(\AA_K^{\infty_K}).$ The product topology of $\prod_{v\in \infty_K}\{$connected components of $C(K_v)\}$
	with discrete topology and $C(\AA_K^{\infty_K})$ with adelic topology, gives a topology for $C(\AA_K)_\bullet.$  For any $A\in\Br(C),$ and any $v\in\infty_K,$ the evaluation of $A$ on each connected component of $C(K_v)$ is constant. So, the notation $C(\AA_K)_\bullet^{\Br}$ makes sense.

	\begin{conjecture}\cite[Conjecture 9.1]{St07}\label{conjecture Stoll}
		For any smooth, projective, and geometrically connected curve $C$ defined over a number field $K,$  the set $C(K)$ is dense in $C(\AA_K)_\bullet^{\Br}.$ In particular, the curve $C$ satisfies weak approximation with Brauer--Manin obstruction off $\infty_K.$
	\end{conjecture}

	\begin{remark}
		For an elliptic curve defined over $K,$ if its Tate--Shafarevich group is finite, then by the dual sequence of Cassels--Tate, Conjecture \ref{conjecture Stoll} holds for this elliptic curve, cf. \cite[Chapter 6.2]{Sk01}.
		With the effort of Kolyvagin \cite{Ko90,Ko91}, Gross and Zagier \cite{GZ86}, and many others, for an elliptic curve $E$ over $\QQ,$ if its analytic rank equals zero or one, then its Mordell-Weil rank equals its analytic rank, and its Tate--Shafarevich  group $\Sha(E,\QQ)$ is finite. So Conjecture \ref{conjecture Stoll} holds for $E.$
		%Original statement for weak approximation with Brauer--Manin obstruction is that $C(K)$ is dense in $C(\AA_K)_\bullet,$ here $C(\AA_K)_\bullet$ is obtained from $C(\AA_K)$ in the way of replacing archimedean places part by product of their sets of connected components with discrete topology. 
	\end{remark}

	%In this paper, we make the following assumptions for the given number field extension $L$ of $K.$
	
	\begin{definition}\label{definition curve of type}
		Given a nontrivial extension of number fields $L/K,$ let $C$ be a smooth, projective, and geometrically connected curve defined over $K.$ We say that a triple $(C,K,L)$  is of \defi{type $I$} if $C(K)$ and $C(L)$ are both finite nonempty sets, $C(K)\neq C(L)$ and Stoll's Conjecture \ref{conjecture Stoll} holds for the curve $C.$ 
	\end{definition}

	\begin{lemma}\label{lemma stoll conjecture}
		Given a nontrivial extension of number fields $L/K,$ if Conjecture \ref{conjecture Stoll} holds for all smooth, projective, and geometrically connected curves defined over $K,$ then there exists a curve $C$ defined over $K$ such that the triple $(C,K,L)$ is of type $I.$
	\end{lemma}
	
	\begin{proof}
		Since $L$ is a finite separable extension over $K,$ there exists a $\theta\in L$ such that $L=K(\theta).$ Let $f(x)$ be the monic minimal polynomial of $\theta.$ Let $n=\deg(f),$ then $n=[L:K]\geq 2.$ Let $\tilde{f}(w_0,w_1)$ be the homogenization of $f.$ If $n$ is odd, we consider a curve $C$ defined over $K$ by a homogeneous equation: $w_2^{n+2}=\tilde{f}(w_0,w_1)(w_1^2-w_0^2)$ with homogeneous coordinates $(w_0:w_1:w_2)\in \PP^2.$ For the polynomials $f(x)$ and $x^2-1$ are separable and coprime in $K[x],$ the curve $C$ is smooth, projective, and geometrically connected. By genus formula for a plane curve, the genus of  $C$ equals $g(C)=n(n+1)/2>1.$ By Faltings's theorem \cite[Satz 7]{Fa83}, the sets $C(K)$ and $C(L)$ are both finite. It is easy to check that $(w_0:w_1:w_2)=(1:1:0)\in C(K)$ and $(\theta:1:0)\in C(L)\backslash C(K).$ By the assumption that Conjecture \ref{conjecture Stoll} holds for all smooth, projective, and geometrically connected curves over $K,$ we have that the triple $(C,K,L)$ is of type $I.$ 
	\end{proof}
	
	\begin{remark}
		%	Let $L=K(\theta),$ and let $f(x)$ be the minimal polynomial of $\theta$ and $f(w_0,w_1)$ be its homogenization. Consider a curve $C$ defined over $K$ by a homogeneous equation: $w_2^{n+2}=f(w_0,w_1)(w_1^2-w_0^2)$ with homogeneous coordinates $(w_0:w_1:w_2)\in \PP^2.$ For the polynomials $f(x)$ and $x^2-1$ are separable and coprime in $K[x],$ this defines a smooth, projective, and geometrically connected curve. By genus formula for plane curve, the genus of  $C$ equals $g(C)=\frac{n(n+1)}{2}\geq 3,$ then by Faltings's theorem, the sets $C(K)$ and $C(L)$ are both finite. For $(w_0:w_1:w_2)=(1:1:0)\in C(K)$ and $(1:\theta:0)\in C(L)\backslash C(K),$ if Stoll's Conjecture \ref{conjecture Stoll} holds for all smooth, projective, and geometrically connected curves over $K,$ then this curve $C$ satisfies our Stoll's conjecture. Furthermore, if $n$ is odd, then $C_{K_{v'}}$ is connected for all $v'\in \infty_L.$ If $n$ is even, then consider a curve $C'$ defined by a homogeneous equation: $w_2^{n+3}=f(w_0,w_1)w_1(w_1^2-w_0^2)$ with homogeneous coordinates $(w_0:w_1:w_2)\in \PP^2.$ Then $C'(L_{v'})$ is connected for all	$v'\in \infty_L.$ In other words,  if $L$ has a real place and Stoll's Conjecture \ref{conjecture Stoll} holds for all smooth, projective, and geometrically connected curves over $K,$ then there exists a curve satisfying our Stoll's conjecture. 	
		For some nonsquare integer $d,$ let $K=\QQ$ and $L=\QQ(\sqrt{d}).$  Consider an elliptic curve $E_d$ defined by a Weierstra\ss~ equation: $y^2=x^3+d.$ 
		Let $E_d^{(d)}$ be the quadratic twist of $E_d$ by $d.$ It is easy to check that the point $(x,y)=(0,\sqrt{d})\in C(L)\backslash C(K).$
		% with Weierstra\ss~ equation: $y^2=x^3+d^4.$ 
		If both $E_d(\QQ)$ and $E_d^{(d)}(\QQ)$ are finite, then the set $E_d(L)$ is finite, cf. \cite[Exercise 10.16]{Si09}. If additionally, the Tate--Shafarevich group $\Sha(E_d,\QQ)$ is finite, then the triple $(E_d,K,L)$ is of type $I.$  
	\end{remark}

	\section{Main results for Ch\^atelet surface bundles over curves}

	\subsection{Preparation Lemmas} We state the following lemmas, which will be used for the proof of our theorems. 
	
	Fibration methods are used in \cite[Proposition 3.1]{CTX13}, \cite[Lemma 5.1]{CX18} and \cite[Section 4]{CLX18}. We modify those fibration methods to fit into our context. 
	%The following two fibration Lemmas  will be used for the proof of our main theorems. In this subsection, we give two fibration lemmas. 
	\begin{lemma}\label{lemma fiber criterion for wabm}
		Let $K$ be a number field,  and let $S\subset \Omega_K$ be a finite subset.  Let $f\colon X\to Y$ be a $K$-morphism of proper $K$-varieties $X$ and $Y.$ We assume that
		\begin{enumerate}{
				\item\label{fiber criterion for wabm condition 1}  the set $Y(K)$ is finite,
				\item\label{fiber criterion for wabm condition 2}  the variety $Y$ satisfies weak approximation with Brauer--Manin obstruction off $S,$
				\item\label{fiber criterion for wabm condition 3}  for any $P\in Y(K),$ the fiber $X_P$ of $f$ over $P$ satisfies weak approximation off $S.$}
		\end{enumerate}
		Then $X$ satisfies weak approximation with Brauer--Manin obstruction off $S.$
	\end{lemma}
	
	\begin{proof}
		For any finite subset $S'\subset\Omega_K\backslash S,$ take an open subset $N=\prod_{v\in S'}U_v\times \prod_{v\notin S'}X(K_v)\subset X(\AA_K)$ such that $N\bigcap X(\AA_K)^{\rm Br}\neq \emptyset.$ Let $M=\prod_{v\in S'}f(U_v)\times \prod_{v\notin S'}f(X(K_v)),$ then by the functoriality of Brauer--Manin pairing, $M\bigcap Y(\AA_K)^{\rm Br}\neq \emptyset.$ By Assumptions (\ref{fiber criterion for wabm condition 1}) and (\ref{fiber criterion for wabm condition 2}), we have $Y(K)= pr^S(Y(\AA_K)^{\rm Br}).$ So there exists $P_0\in pr^S(M)\bigcap Y(K).$ Consider the fiber $X_{P_0}.$ Let $L=\prod_{v\in S'} [X_{P_0}(K_v)\bigcap U_v]\times \prod_{v\notin S'\cup S} X_{P_0}(K_v),$ then it is a nonempty open subset of $X_{P_0}(\AA_K^S).$  By Assumption (\ref{fiber criterion for wabm condition 3}), there exists $Q_0\in L\bigcap X_{P_0}(K).$ So $Q_0\in X(K)\bigcap N,$ which implies that $X$ satisfies weak approximation with Brauer--Manin obstruction off $S.$
	\end{proof}

	\begin{lemma}\label{lemma fiber criterion for not wabm}
		Let $K$ be a number field,  and let $S\subset \Omega_K$ be a finite subset.  Let $f\colon X\to Y$ be a $K$-morphism of proper $K$-varieties $X$ and $Y.$ We assume that
		\begin{enumerate}{
				\item\label{fiber criterion for not wabm condition 1} the set $Y(K)$ is finite,
				\item\label{fiber criterion for not wabm condition 2} the morphism $f^*\colon \Br(Y)\to \Br(X)$ is surjective,
				\item\label{fiber criterion for not wabm condition 3} there exists some $P\in Y(K)$ such that the fiber $X_P$ of $f$ over $P$ does not satisfy weak approximation off $S,$ and that $\prod_{v\in S}X_P(K_v)\neq \emptyset.$  }
		\end{enumerate}
		Then $X$ does not satisfy weak approximation with Brauer--Manin obstruction off $S.$
	\end{lemma}
	
	\begin{proof}	
		By Assumption (\ref{fiber criterion for not wabm condition 3}), take a $P_0\in Y(K)$ such that the fiber $X_{P_0}$ does not satisfy weak approximation  off $S,$ and that $\prod_{v\in S} X_{P_0}(K_v)\neq \emptyset.$ Then there exist a finite nonempty subset $S'\subset\Omega_K\backslash S$ and a nonempty open subset $L=\prod_{v\in S'}U_v\times \prod_{v\notin S'}X_{P_0}(K_v)\subset X_{P_0}(\AA_K)$ such that $L\bigcap X_{P_0}(K)=\emptyset.$ 
		By Assumption (\ref{fiber criterion for not wabm condition 1}), the set $Y(K)$ is finite, so we can take a Zariski open subset $V_{P_0}\subset Y$ such that $V_{P_0}(K)=\{P_0\}.$ For any $v\in S',$ since $U_v$ is open in $X_{P_0}(K_v)\subset f^{-1}(V_{P_0})(K_v),$ we can take an open subset $W_v$ of $f^{-1}(V_{P_0})(K_v)$ such that $W_v\cap X_{P_0}(K_v)=U_v.$ 	
		Consider the open subset $N=\prod_{v\in S'}W_v\times \prod_{v\notin S'}X(K_v)\subset X(\AA_K),$ then $L\subset N.$ By the functoriality of Brauer--Manin pairing and Assumption (\ref{fiber criterion for not wabm condition 2}), we have $L\subset N\bigcap X(\AA_K)^{\Br}.$ So $N\bigcap X(\AA_K)^{\Br} \neq \emptyset.$ But $N\bigcap X(K)=N\bigcap X_{P_0}(K)=L\bigcap X_{P_0}(K)
		=\emptyset,$ which implies that $X$ does not satisfy weak approximation with Brauer--Manin obstruction off $S.$
	\end{proof}

	We use the following lemma to choose a dominant morphism from a given curve to $\PP^1.$

	\begin{lemma}\label{lemma choose base change morphism}
		Given a nontrivial extension of number fields $L/K,$ let $C$ be a smooth, projective, and geometrically connected curve defined over $K.$ Assume that the triple $(C,K,L)$ is of type $I$ (Definition \ref{definition curve of type}). For any finite $K$-subscheme $R\subset \PP^1\backslash\{0,\infty\},$  there exists a dominant $K$-morphism $\gamma\colon  C\to \PP^1$ such that $\gamma(C(L)\backslash C(K))=\{0\}\subset \PP^1(K),$ $\gamma(C(K))=\{\infty\}\subset \PP^1(K),$ and that $\gamma$ is \'etale over $R.$ 
	\end{lemma}
	
	\begin{proof}
		Let $K(C)$ be the function field of $C.$ % Let $M$ be the smallest Galois extension containing $L$ of $K,$ and let $\Gal(M/K)$ be the Galois group.
		For $C(K)$ and $C(L)$ are both finite nonempty sets and $C(L)\backslash C(K)\neq \emptyset,$ by Riemann-Roch theorem, we can choose a rational function $\phi\in K(C)^\times\backslash K^\times$ such that the set of its poles contains $C(K),$ and that the set of its zeros contains $C(L)\backslash C(K).$ %Replacing $\phi$ by $\prod_{\sigma\in \Gal(M/K)} \sigma(\phi),$ we can assume $\phi\in K(C)^\times\backslash K^\times.$
		This rational function $\phi$ gives a dominant $K$-morphism $\gamma_0\colon C\to \PP^1$ such that $\gamma_0(C(L)\backslash C(K))=\{0\}\subset \PP^1(K)$ and $\gamma_0(C(K))=\{\infty\}\subset \PP^1(K).$  We can choose an automorphism $\varphi_{\lambda_0}\colon \PP^1\to \PP^1, (u:v)\mapsto (\lambda_0 u:v)$ with $\lambda_0\in K^\times$ such that the branch locus of $\gamma_0$ has no intersection with $\varphi_{\lambda_0}(R).$ Let $\gamma= (\varphi_{\lambda_0})^{-1}\circ\gamma_0.$ Then the morphism $\gamma$ is \'etale over $R$ and satisfies other conditions. 
	\end{proof}

	The following lemma is well known.
	
	\begin{lemma}\label{lemma zero of Br B}
		Let $C$ be a curve over a number field, and let $B=C\times \PP^1.$ Then $\Br \overline{B}=0.$ 
	\end{lemma}
	
	\begin{proof}
		By \cite[III, Corollary 1.2]{Gr68}, the Brauer group for a given curve over an algebraic closed field is zero. So $\Br(\overline{C}\times \overline{\PP^1})\cong \Br(\overline{C})=0.$
	\end{proof}

	\begin{definition}
		Let $C$ be a smooth, projective, and geometrically connected curve defined over a number field. We say that a morphism
		$\beta\colon X\to C$ is a \defi{Ch\^atelet surface bundle over the curve} $C,$ if
		\begin{itemize}
			\item $X$ is a smooth, projective, and geometrically connected variety,
			\item the morphism $\beta$ is faithfully flat and proper,
			\item the generic fiber of $\beta$ is a Ch\^atelet surface over the function field of $C.$
		\end{itemize}
	\end{definition}

	Next, we construct Ch\^atelet surface bundles over curves to give a negative answer to Question \ref{Question}.

	\subsection{Non-invariance of weak approximation with Brauer--Manin obstruction}
	
	For any number field $K,$ assuming Conjecture \ref{conjecture Stoll}, Liang \cite[Theorem 4.5]{Li18} found a quadratic extension $L,$ and constructed  a Ch\^atelet surface bundle over a curve to give a negative answer to Question \ref{question on WA1}. For a given number field $K,$ by choosing prime elements, Liang found a quadratic extension $L,$ and constructed a Ch\^atelet surface defined over $K$ such that the property of weak approximation is not invariant under the extension $L/K.$ Then choosing a higher genus curve, Liang combined this Ch\^atelet surface with the construction method of Poonen \cite{Po10} to get Liang's result. In this subsection, we generalize Liang's result for quadratic extensions to any nontrivial extension of number fields $L/K.$ 
	
	\begin{theorem}\label{theorem main result: non-invariance of weak approximation with BMO}
		For any nontrivial extension of number fields $L/K,$ and any finite subset $T\subset \Omega_L,$ assuming that Conjecture \ref{conjecture Stoll} holds over $K,$ there exist a curve $C$ and a Ch\^atelet surface bundle: $X\to C$  defined over $K$ such that
		\begin{itemize}
			\item $X$ has a $K$-rational point, and satisfies weak approximation with Brauer--Manin obstruction %, \'etale Brauer--Manin obstruction
			off $\infty_K,$
			\item $X_L$ does not satisfy weak approximation with Brauer--Manin obstruction %, \'etale Brauer--Manin obstruction
			off $T.$
		\end{itemize}
	\end{theorem}
	
	\begin{proof}
		Firstly, we construct two Ch\^atelet surfaces.  Let $T_0=T\cup \infty_L.$ By Proposition \ref{proposition the valuation of Brauer group on local points are fixed outside S and take two value on S}, let $V_0$ be a Ch\^atelet surface defined by some equation $y^2-az^2=P_0(x)$ over $K$ has the properties of Proposition \ref{proposition the valuation of Brauer group on local points are fixed outside S and take two value on S}. 	Let $P_\infty(x)=(1-x^2)(x^2-a),$ and let $V_\infty$ be the Ch\^atelet surface defined by $y^2-az^2=P_\infty(x).$ Choose any polynomial $P_0(x)\in K[x]$ coprime to $P_\infty(x)$ which exists as in the proof of Proposition \ref{proposition the valuation of Brauer group on local points are fixed outside S and take two value on S}.
		
		Secondly, we construct a Ch\^atelet surface bundle over a curve.
		Let $\tilde{P}_\infty(x_0,x_1)$ and $\tilde{P}_0(x_0,x_1)$ be the homogenizations of $P_\infty(x)$ and $P_0(x).$ Let $(u_0:u_1)\times(x_0:x_1)$ be the bi-homogeneous coordinates of $\PP^1\times\PP^1,$ and let $s'=u_0^2\tilde{P}_\infty(x_0,x_1)+u_1^2\tilde{P}_0(x_0,x_1)\in \Gamma(\PP^1\times\PP^1,\Ocal(1,2)^{\otimes2}).$ For $P_0(x)$ and $P_\infty(x)$ are coprime in $K[x],$ by Jacobian criterion, the locus $Z'$ defined by $s'=0$ in $\PP^1\times\PP^1$ is smooth. Then the branch locus of the composition $ Z'\hookrightarrow \PP^1\times\PP^1  \stackrel{pr_1}\to\PP^1,$ denoted by $R,$ is finite over $K.$ 
		Since Conjecture \ref{conjecture Stoll} holds over $K$ by assumption, using Lemma \ref{lemma stoll conjecture}, we can take a curve $C$ defined over $K$ such that the triple $(C,K,L)$ is of type $I.$ By Lemma \ref{lemma choose base change morphism}, we can choose a $K$-morphism $\gamma\colon  C\to \PP^1$ such that $\gamma(C(L)\backslash C(K))=\{0\}\subset \PP^1(K),$ $\gamma(C(K))=\{\infty\}\subset \PP^1(K),$ and that $\gamma$ is \'etale over $R.$ 
		Let $B=C\times \PP^1,$ and let $(\gamma,id)\colon B\to\PP^1\times \PP^1.$  Let $\Lcal=(\gamma,id)^*\Ocal(1,2),$ and let $s=(\gamma,id)^* (s')\in \Gamma(B,\Lcal^{\otimes2}).$ For $\gamma$ is \'etale over the branch locus  $R,$  the locus $Z$ defined by $s=0$ in $B$ is smooth. Since $Z$ is defined by the support of the global section $s,$ it is an effective divisor. The invertible sheaf $\Lscr (Z')$ on $\PP^1\times\PP^1$ is isomorphic to $\Ocal(2,4),$ which is a very ample sheaf on $\PP^1\times\PP^1.$ And $(\gamma,id)$ is a finite morphism, so the pull back of this ample sheaf is again ample, which implies that the invertible sheaf $\Lscr (Z)$ on $C\times\PP^1$ is ample. By \cite[Chapter III. Corollary 7.9]{Ha97}, the curve $Z$ is geometrically connected. So the curve $Z$ is smooth, projective, and geometrically connected. By Lemma \ref{lemma zero of Br B}, the Brauer group $\Br(\overline{B})=0.$  Let $X$ be the zero locus of $(1,-a,-s)\in\Gamma(B,\Ocal_B\oplus\Ocal_B\oplus\Lcal^{\otimes2})\subset\Gamma(B,\Sym^2\Ecal)$ in the projective space bundle $\Proj(\Ecal)$ with the natural projection $\alpha\colon  X\to B.$	 Using Proposition \ref{Poonen's main proposition}, the variety $X$ is smooth, projective, and geometrically connected.
		Let $\beta\colon X \stackrel{\alpha}\to B=C\times \PP^1 \stackrel{pr_1}\to C$ be the composition of $\alpha$ and $pr_1.$ Then the morphism $\beta$ is a Ch\^atelet surface bundle over the curve $C.$
		
		At last, we check that the variety $X$ has the properties.
		
		We show that  $X$ has a $K$-rational point. For any $P\in C(K),$ the fiber  $\beta^{-1}(P)\cong V_\infty.$ The surface $V_\infty$ has a $K$-rational point $(x,y,z)=(0,0,1),$ so the set $X(K)\neq \emptyset.$
		%For $\Br(V_\infty)/\Br(K)=0,$ according to \cite[Theorem B]{CTSSD87a,CTSSD87b},
		
		We show that  $X$ satisfies weak approximation with Brauer--Manin obstruction %, \'etale Brauer--Manin obstruction 
		off $\infty_K.$ 
		By Remark \ref{remark birational to Hasse--Minkowski theorem}, the surface $V_\infty$ satisfies weak approximation. So, for the morphism $\beta,$ 
		Assumption (\ref{fiber criterion for wabm condition 3}) of Lemma \ref{lemma fiber criterion for wabm} holds.
		Since Conjecture \ref{conjecture Stoll} holds for the curve $C,$ using Lemma \ref{lemma fiber criterion for wabm} for the morphism $\beta,$ the variety $X$ satisfies weak approximation with Brauer--Manin obstruction %, \'etale Brauer--Manin obstruction 
		off $\infty_K.$ 
		
		We show that  $X_L$ does not satisfy weak approximation with Brauer--Manin obstruction %, \'etale Brauer--Manin obstruction 
		off $T.$	
		By Proposition \ref{Poonen's main proposition}, the map $\alpha_L^*\colon\Br(B_L)\to \Br(X_L)$ is an isomorphism, so $\beta_L^*\colon\Br(C_L)\to \Br(X_L)$ is an isomorphism. By the choice of the curve $C$ and morphism $\beta,$ for any $Q\in C(L)\backslash C(K),$ the fiber $\beta^{-1}(Q)\cong V_{0L}.$ 
		By Proposition \ref{proposition the valuation of Brauer group on local points are fixed outside S and take two value on S}, the surface $V_{0L}$ does not satisfy weak approximation off $T\cup \infty_L.$ For $V_0(L)\neq\emptyset,$ by Lemma \ref{lemma fiber criterion for not wabm}, the variety $X_L$ does not satisfy weak approximation with Brauer--Manin obstruction %, \'etale Brauer--Manin obstruction 
		off $T\cup \infty_L.$ So it does not satisfy weak approximation with Brauer--Manin obstruction %, \'etale Brauer--Manin obstruction 
		off $T.$
	\end{proof}

	\section{An Explicit unconditional example}\label{section main example1}

	Let $K=\QQ$ and $L=\QQ(\sqrt{3}).$	In this section, we give an explicit unconditional (without assuming Conjecture \ref{conjecture Stoll}) example  for Theorem \ref{theorem main result: non-invariance of weak approximation with BMO}.

	\subsection{Choosing an elliptic curve} 
	Let $E$ be an elliptic curve defined over $\QQ$ by a homogeneous equation:
	$$w_1^2w_2=w_0^3-16w_2^3$$ with homogeneous coordinates $(w_0:w_1:w_2)\in \PP^2.$ This is an elliptic curve with complex multiplication. Its quadratic twist $E^{(3)}$ is isomorphic to an elliptic curve defined by a homogeneous equation:
	$w_1^2w_2=w_0^3-432w_2^3$ with homogeneous coordinates $(w_0:w_1:w_2)\in \PP^2.$ These elliptic curves $E$ and $E^{(3)}$ defined over $\QQ,$ are  of analytic rank $0.$ Then the Tate--Shafarevich group $\Sha(E,\QQ)$ is finite, so $E$ satisfies weak approximation with Brauer--Manin obstruction off $\infty_K.$ The Mordell-Weil groups $E(K)$ and $E^{(3)}(K)$ are both finite, so $E(L)$ is finite. Indeed, the Mordell-Weil groups $E(K)=\{(0:1:0)\}$ and $E(L)=\{(4:\pm 4\sqrt{3}:1),(0:1:0)\}.$ So the triple $(E,K,L)$ is of type $I.$

	%Then the Tate--Shafarevich group $\Sha(E,\QQ)$ and $\Sha(E^{(3)},\QQ)$ are both finite. The curves $E_K$ and $E_L$ satisfy weak approximation with Brauer--Manin obstruction off $\infty_K$ and $\infty_L$ respectively. The Mordell-Weil group $E(K)$ and $E(L)$ are both finite. Indeed,  the Mordell-Weil group $E(K)=\{(0:1:0)\}$ and $E(L)=\{(4:\pm 4\sqrt{3}:1),(0:1:0)\}.$ 

	\subsection{Choosing a dominant morphism}
	Let $\PP^2\backslash\{(0:1:0)\}\to \PP^1$ be a morphism over $\QQ$ given by $(w_0:w_1:w_2)\mapsto(w_0-4w_2:w_2).$ Composing the natural inclusion $E\backslash\{(0:1:0)\}\hookrightarrow \PP^2\backslash\{(0:1:0)\}$ with it, we get a morphism $E\backslash\{(0:1:0)\}\to \PP^1,$ which can be extended to a dominant morphism $\gamma\colon E\to \PP^1$  of degree $2.$ One can check that the morphism $\gamma$ maps $E(K)$ to $\{\infty\}=\{(1:0)\},$ and maps $(4:\pm 4\sqrt{3}:1)$ to $0$ point: $(0:1).$	By B\'ezout's Theorem \cite[Chapter I. Corollary 7.8]{Ha97} or Hurwitz's Theorem \cite[Chapter IV. Corollary 2.4]{Ha97}, the branch locus of $\gamma$ is $\{(1:0),(2\sqrt[3]{2}-4:1),(2\sqrt[3]{2}e^{2\pi i/3}-4:1),(2\sqrt[3]{2}e^{-2\pi i/3}-4:1)\}.$

	\subsection{Construction of a Ch\^atelet surface bundle}

	Let $P_\infty(x)=(1-x^2)(x^2-73),$ and let $P_0(x)=(5x^2+1)(5334x^2/5329+1/5329).$ Notice that these polynomials $P_\infty(x)$ and $P_0(x)$ are separable.
	Let $V_\infty$ be the Ch\^atelet surface given by $y^2-73z^2=P_\infty(x).$ As mentioned in Example \ref{example1: construction of V_0}, let $V_0$ be the Ch\^atelet surface given by $y^2-73z^2=P_0(x).$ 
	Let $\tilde{P}_\infty(x_0,x_1)$ and $\tilde{P}_0(x_0,x_1)$ be the homogenizations of $P_\infty(x)$ and $P_0(x).$ Let $(u_0:u_1)\times(x_0:x_1)$ be the bi-homogeneous coordinates of $\PP^1\times\PP^1,$ and let $s'=u_0^2\tilde{P}_\infty(x_0,x_1)+u_1^2\tilde{P}_0(x_0,x_1)\in \Gamma(\PP^1\times\PP^1,\Ocal(1,2)^{\otimes2}).$ For $P_0(x)$ and $P_\infty(x)$ are coprime in $K[x],$ by Jacobian criterion, the locus $Z'$ defined by $s'=0$ in $\PP^1\times\PP^1$ is smooth. Then the branch locus of the composition $ Z'\hookrightarrow \PP^1\times\PP^1  \stackrel{pr_1}\to\PP^1,$ denoted by $R,$ is finite, and contained in $\PP^1\backslash \{(1:0)\}.$
	Let $B=E\times \PP^1,$ and let $(\gamma,id)\colon B\to\PP^1\times \PP^1.$ Let $\Lcal=(\gamma,id)^*\Ocal(1,2),$ and let $s=(\gamma,id)^* (s')\in \Gamma(B,\Lcal^{\otimes2}).$ With the notation, we have the following lemma.

	\begin{lemma}
		The curve $Z$ defined by $s=0$ in $B$ is smooth, projective, and geometrically connected.
	\end{lemma}
	
	\begin{proof}
		For smoothness of $Z,$ we need to check that the branch locus $R$ %of $Z'\hookrightarrow \PP^1\times\PP^1  \stackrel{pr_1}\to\PP^1$ 
		does not intersect with the branch locus of  $\gamma\colon E\to \PP^1.$ %The first morphism $ Z'\to\PP^1$ is a dominant morphism of degree $4,$ and its branch locus 
		For $R$ is contained in $\PP^1\backslash \{(1:0)\},$ % As before, let $(u_0:u_1)$ be the homogeneous coordinates of $\PP^1.$ 	
		we can assume the homogeneous coordinate $u_1=1,$ then the point $(u_0:1)$ in $R$ satisfies one of the following equations: 
		$5329u_0^2-26670=0,~ 389017u_0^2 - 1=0,~27625536u_0^4 + 8577816u_0^2 + 5329=0.$ The polynomials of these equations are irreducible over $\QQ.$  By comparing the degree $[\QQ(u_0):\QQ]$ with the branch locus of $\gamma,$ we get the conclusion that these two branch loci do not intersect.
		The same argument as in the proof of Theorem \ref{theorem main result: non-invariance of weak approximation with BMO},	
		the locus $Z$ defined by $s=0$ in $B$ is geometrically connected. So it is smooth, projective, and geometrically connected.
	\end{proof}

	Let $X$ be the zero locus of $(1,-a,-s)\in\Gamma(B,\Ocal_B\oplus\Ocal_B\oplus\Lcal^{\otimes2})\subset\Gamma(B,\Sym^2\Ecal)$ in the projective space bundle $\Proj(\Ecal)$ with the natural projection $\alpha\colon  X\to B.$  By the same argument as in the proof of Theorem \ref{theorem main result: non-invariance of weak approximation with BMO}, the variety $X$ is smooth, projective, and geometrically connected. Let $ \beta\colon X \to E$ be the composition of $\alpha$ and $pr_1.$ Then it is a Ch\^atelet surface bundle over the curve $E.$ For this variety $X,$ we have the following proposition.

	\begin{prop}\label{example1: main result: non-invariance of weak approximation with BMO}
		For $K=\QQ$ and $L=\QQ(\sqrt{3}),$  the $3$-fold $X$ has the following properties.	
		\begin{itemize}
			\item $X$ has a $K$-rational point, and satisfies weak approximation with Brauer--Manin obstruction %, \'etale Brauer--Manin obstruction 
			off $\infty_K.$
			\item $X_L$ does not satisfy  weak approximation with Brauer--Manin obstruction %, \'etale Brauer--Manin obstruction 
			off $\infty_L.$
		\end{itemize}
	\end{prop}

	\begin{proof}
		This is the same as in the proof of Theorem \ref{theorem main result: non-invariance of weak approximation with BMO}.
	\end{proof}

	The $3$-fold $X$ that we constructed, has an affine open subvariety defined by the following equations, which is a closed subvariety of $\AA^5$ with affine coordinates $(x,y,z,x',y').$
	\begin{equation*}
		\begin{cases}
			y^2-73z^2=(1-x^2)(x^2-73)(x'-4)^2+(5x^2+1)(5334x^2/5329+1/5329)\\ 
			{y'}^2={x'}^3-16
		\end{cases}.
	\end{equation*}

	%
	%\begin{proof}
	%	By the same argument as in the proof of Theorem \ref{theorem main result: non-invariance of the Hasse principle with BMO for odd degree}, the variety $X$ is smooth, projective, and geometrically connected. 
	%	For the surface $V_\infty$ has a $K$-rational point $(x,y,z)=(0,2,0),$ the set $X(K)\neq \emptyset.$
	%	For $\Br(V_\infty)/\Br(K)=0,$ according to \cite[Theorem B]{CTSSD87a,CTSSD87b}, the surface $V_\infty$ satisfies weak approximation.  Because the curve $E$ satisfies Conjecture \ref{conjecture Stoll}, using Lemma \ref{lemma fiber criterion for wabm} for the morphism $\beta,$ the variety $X$ has a $K$-rational point, and satisfies weak approximation with Brauer--Manin obstruction, \'etale Brauer--Manin obstruction off $\infty_K.$ 
	%	
	%	By Proposition \ref{Poonen's main proposition}, we have $\alpha_L^*\colon\Br(B_L)\to \Br(X_L)$ is an isomorphism. By construction, we have $\beta^{-1}(E(L))=V_{\infty L}\bigcup V_{0 L}\bigcup V_{0 L}.$ 
	%	By Example \ref{example1: construction of V_0} and Proposition \ref{proposition the valuation of Brauer group on local points are fixed outside S and take two value on S}, the surface $V_{0L}$ does not satisfy weak approximation off $\infty_L$ and the set $V_0(L)\neq\emptyset.$ By Lemma \ref{lemma fiber criterion for not wabm}, the variety $X$ does not satisfy weak approximation with Brauer--Manin obstruction, \'etale Brauer--Manin obstruction off $\infty_L.$ 
	%\end{proof}

	\begin{footnotesize}
		\noindent\textbf{Acknowledgements.} The author would like to thank his thesis advisor Y. Liang for proposing the related problems, papers and many fruitful discussions, and thank  the anonymous
		referees for their careful scrutiny and valuable suggestions.  The author is partially supported by NSFC Grant No. 12071448.
	\end{footnotesize}
	
	

\begin{bibdiv}
		\begin{biblist}
			
			\bib{CLX18}{article}{
				author={Cao, Y.},
				author={Liang, Y.},
				author={Xu, F.},
				title={Arithmetic purity of strong approximation for homogeneous
					spaces},
				date={2018},
				journal={Preprint, arXiv:1701.07259v3 [math.AG]},
			}
			
			\bib{CT03}{book}{
				author={Colliot-Th\'el\`ene, J.-L.},
				title={Points rationnels sur les fibrations. {I}n},
				subtitle={Higher dimensional varieties and rational points},
				series={Bolyai Society Mathematical studies},
				publisher={Springer-Verlag},
				date={2003},
				volume={12},
				note={pp. 171-221},
			}
			
			\bib{CT10}{article}{
				author={Colliot-Th\'el\`ene, J.-L.},
				title={Z\'ero-cycles de degr\'e 1 sur les solides de {P}oonen},
				language={French},
				date={2010},
				journal={Bull. Soc. Math. France},
				volume={138},
				pages={249\ndash 257},
			}
			
			\bib{CTSSD87a}{article}{
				author={Colliot-Th\'el\`ene, J.-L.},
				author={Sansuc, J.-J.},
				author={Swinnerton-Dyer, S.},
				title={Intersections of two quadrics and {C}h\^atelet surfaces {I}},
				date={1987},
				journal={J. Reine Angew. Math.},
				volume={373},
				pages={37\ndash 107},
			}
			
			\bib{CTSSD87b}{article}{
				author={Colliot-Th\'el\`ene, J.-L.},
				author={Sansuc, J.-J.},
				author={Swinnerton-Dyer, S.},
				title={Intersections of two quadrics and {C}h\^atelet surfaces {II}},
				date={1987},
				journal={J. Reine Angew. Math.},
				volume={374},
				pages={72\ndash 168},
			}
			
			\bib{CTX13}{article}{
				author={Colliot-Th\'el\`ene, J.-L.},
				author={Xu, F.},
				title={Strong approximation for the total space of certain quadric
					brations},
				date={2013},
				journal={Acta Arith.},
				volume={157},
				pages={169\ndash 199},
			}
			
			\bib{CX18}{article}{
				author={Cao, Y.},
				author={Xu, F.},
				title={Strong approximation with brauer-manin obstruction for toric
					varieties},
				date={2018},
				journal={Ann. Inst. Fourier (Grenoble)},
				volume={5},
				pages={1879\ndash 1908},
			}
			
			\bib{Fa83}{article}{
				author={Faltings, G.},
				title={Endlichkeitss{\"a}tze f{\"u}r abelsche {V}ariet{\"a}ten {\"u}ber
					{Z}ahlk{\"o}rpern},
				date={1983},
				journal={Invent. math.},
				volume={73},
				pages={349\ndash 366},
			}
			
			\bib{Gr68}{book}{
				author={Grothendieck, A.},
				title={Le groupe de {B}rauer {III}: {E}xemples et compl\'ements. {I}n},
				subtitle={Dix expos\'es sur la cohomologie des sch\'emas},
				language={French},
				series={Advanced Studies in Pure Mathematics},
				publisher={North-Holland},
				date={1968},
				volume={3},
				note={pp. 88-188},
			}
			
			\bib{GZ86}{article}{
				author={Gross, B.},
				author={Zagier, D.},
				title={Heegner points and derivatives of {L}-series},
				date={1986},
				journal={Invent. Math.},
				volume={84},
				pages={225\ndash 320},
			}
			
			\bib{Ha97}{book}{
				author={Hartshorne, R.},
				title={Algebraic geometry},
				series={Graduate Texts in Mathematics},
				publisher={Springer-Verlag},
				date={1997},
				volume={52},
			}
			
			\bib{Ko90}{book}{
				author={Kolyvagin, V.},
				title={Euler systems. {I}n},
				subtitle={The {G}rothendieck festschrift {II}},
				series={Progress in Mathematics},
				publisher={Birkh{\"a}user},
				date={1990},
				volume={87},
				note={pp. 435-483},
			}
			
			\bib{Ko91}{book}{
				author={Kolyvagin, V.},
				title={On the {M}ordell-{W}eil and the {S}hafarevich-{T}ate group of
					modular elliptic curves. {I}n},
				subtitle={Proceedings of the international congress of mathematicians},
				publisher={Springer-Verlag},
				date={1991},
				volume={I},
				note={pp. 429-436},
			}
			
			\bib{Li18}{article}{
				author={Liang, Y.},
				title={Non-invariance of weak approximation properties under extension
					of the ground field},
				date={2018},
				journal={Michigan Math. J.},
			}
			
			\bib{Ma71}{book}{
				author={Manin, Y.},
				title={Le groupe de {B}rauer-{G}rothendieck en g\'eom\'etrie
					diophantienne. {I}n},
				subtitle={Actes du {C}ongr\`es {I}nternational des {M}ath\'ematiciens},
				language={French},
				publisher={Gauthier-Villars},
				date={1971},
				volume={1},
				note={pp. 401-411},
			}
			
			\bib{Po09}{article}{
				author={Poonen, B.},
				title={Existence of rational points on smooth projective varieties},
				date={2009},
				journal={J. Eur. Math. Soc.},
				volume={11},
				pages={529\ndash 543},
			}
			
			\bib{Po10}{article}{
				author={Poonen, B.},
				title={Insufficiency of the {B}rauer-{M}anin obstruction applied to
					\'etale covers},
				date={2010},
				journal={Ann. of Math.},
				volume={171},
				pages={2157\ndash 2169},
			}
			
			\bib{Se79}{book}{
				author={Serre, J.-P.},
				title={Local fields},
				series={Graduate Texts in Mathematics},
				publisher={Springer-Verlag},
				date={1979},
				volume={67},
			}
			
			\bib{Si09}{book}{
				author={Silverman, J.},
				title={The arithmetic of elliptic curves},
				series={Graduate Texts in Mathematics},
				publisher={Springer-Verlag},
				date={2009},
				volume={106},
			}
			
			\bib{Sk01}{book}{
				author={Skorobogatov, A.},
				title={Torsors and rational points},
				series={Cambridge Tracts in Mathematics},
				publisher={Cambridge University Press},
				date={2001},
				volume={144},
			}
			
			\bib{St07}{article}{
				author={Stoll, M.},
				title={Finite descent obstructions and rational points on curves},
				date={2007},
				journal={Algebra Number Theory},
				volume={1},
				pages={349\ndash 391},
			}
			
			\bib{Wa96}{article}{
				author={Wang, L.},
				title={Brauer-{M}anin obstruction to weak approximation on abelian
					varities},
				date={1996},
				journal={Israel J. Math.},
				volume={94},
				pages={189\ndash 200},
			}
			
			\bib{Wu21}{article}{
				author={Wu, H.},
				title={Non-invariance of the {B}rauer-{M}anin obstruction for surfaces},
				date={2021},
				journal={Preprint, arXiv:2103.01784v2 [math.NT]},
			}
			
			\bib{Wu22a}{article}{
				author={Wu, H.},
				title={Arithmetic of {C}h\^atelet surfaces under extensions of base
					fields},
				date={2022},
				journal={Preprint, arXiv:2203.09156 [math.NT]},
			}
			
		\end{biblist}
	\end{bibdiv}
\end{document}